\documentclass[11pt,a4paper]{amsart}
\usepackage{amsfonts,amsmath,amssymb}
\usepackage{hyperref}
\usepackage[capitalize]{cleveref}
\usepackage{latexsym,fullpage,amsfonts,amssymb,amsmath,amscd,graphics}
\usepackage[all]{xy}
\usepackage{amssymb,amsthm,amsxtra}
\usepackage[usenames]{color}
\usepackage{amscd}
\usepackage{amsthm}
\usepackage{amsfonts}
\usepackage{amssymb}
\usepackage{mathrsfs}
\usepackage{url}
\usepackage{bbm}
\usepackage{wasysym}
\usepackage{enumitem}
\usepackage[vcentermath]{youngtab}%
\usepackage{framed}

\usepackage{tikz}
\usetikzlibrary{decorations.markings}
\usetikzlibrary{cd}
\usetikzlibrary{patterns}
\usetikzlibrary{shapes}

\tikzset{anchorbase/.style={baseline={([yshift=-0.5ex]current bounding 
box.center)}}}
\tikzset{wipe/.style={white,line width=4pt}}

\tikzset{->-/.style={decoration={
  markings,
  mark=at position #1 with {\arrow{>}}},postaction={decorate}}}
\tikzset{-<-/.style={decoration={
  markings,
  mark=at position #1 with {\arrow{<}}},postaction={decorate}}}

\tikzset{darkg/.style={green!70!black}}

\tikzset{->-/.style={decoration={
  markings,
  mark=at position #1 with {\arrow{>}}},postaction={decorate}}}
\tikzset{-<-/.style={decoration={
  markings,
  mark=at position #1 with {\arrow{<}}},postaction={decorate}}}

\theoremstyle{plain}
\newtheorem*{theorem*}{Theorem}
\newtheorem*{remark*}{Remark}
\newtheorem*{example*}{Example}
\newtheorem{lemma}{Lemma}[subsection]
\newtheorem{proposition}[lemma]{Proposition}
\newtheorem{corollary}[lemma]{Corollary}

\newtheorem*{conjecture*}{Conjecture}

\theoremstyle{definition}
\newtheorem{definition}[lemma]{Definition}

\newtheorem{example}[lemma]{Example}

\theoremstyle{remark}
\newtheorem{remark}[lemma]{Remark}

\newcommand{\Hom}{\operatorname{Hom}}

\newcommand{\Ind}{\operatorname{Ind}}
\newcommand{\Res}{\operatorname{Res}}

\newcommand{\id}{\operatorname{Id}}

\newcommand{\Aut}{{\operatorname{Aut}}}

\newcommand{\End}{\operatorname{End}}

\newcommand{\C}{{\mathbb C}}

\newcommand{\K}{\mathbb{K}}
\newcommand{\F}{{\mathbb F}}

\newcommand{\abs}[1]{\left|{#1}\right|}

\newcommand{\Rep}{\operatorname{Rep}}

\newcommand{\Vect}[1]{\mathrm{\mathbf{Vect}}_{#1}}

\newcommand{\Sets}{\mathbf{Sets}}

\newcommand{\one}{{\mathbf{1}}}

\newcommand{\rel}[1]{ \langle #1 \rangle}

\newcommand{\comment}[1]{}

\newcommand{\Inna}[1]{\begin{framed} {\tt{\color{blue}{#1}}} \end{framed}}

\def\quotient#1#2{%
    \raise1ex\hbox{$#1$}\Big/\lower1ex\hbox{$#2$}%
}

\begin{document}
\title{Deligne-Knop tensor categories and functoriality}
\date{\today}

  \author{Inna Entova-Aizenbud, Thorsten Heidersdorf}

 \address{I. E.: Department of Mathematics, Ben Gurion University of the Negev, Beer-Sheva, Israel}
 \email{entova@bgu.ac.il}
 \address{T. H.: School of Mathematics, Statistics and Physics, Newcastle University, UK}
 \email{heidersdorf.thorsten@gmail.com}

 \thanks{2010 {\it Mathematics Subject Classification}: 05E05, 18D10, 20C30.}

\begin{abstract} A general construction of Knop creates a symmetric monoidal category $\mathcal{T}(\mathcal{A},\delta)$ from any regular category $\mathcal{A}$ and a fixed degree function $\delta$. A special case of this construction are the Deligne categories $\underline{\Rep}(S_t)$ and $\underline{\Rep}(GL_t(\mathbb{F}_q))$. 
We discuss when a functor $F:\mathcal{A} \to \mathcal{A}'$ between regular categories induces a symmetric monoidal functor $\mathcal{T}(\mathcal{A},\delta) \to \mathcal{T}(\mathcal{A}',\delta')$. We then give a criterion when a pair of adjoint functors between two regular categories $\mathcal{A}, \ \mathcal{A}'$ lifts to a pair of adjoint functors between $\mathcal{T}(\mathcal{A},\delta)$ and $\mathcal{T}(\mathcal{A}',\delta')$.
\end{abstract}

\maketitle

\section{Introduction} 

Many of the known Deligne or interpolating categories arise via a construction of Knop (see \cite{K}) from a regular category $\mathcal{A}$ and a fixed degree function $\delta: Epi(\mathcal{A}) \to \mathbb{K}$ for some fixed ground field $\mathbb{K}$. A typical example is the Deligne category $\underline{\Rep}(S_t)$, $t \in \mathbb{C}$ \cite{Del07} which is obtained from the category $\mathcal{A} = \Sets^{op}$ together with the degree function $\delta (e:Y \to Y') = t^{|Y' \setminus e(Y)|}$ where $e:Y \to Y'$ in $\Sets$ is viewed as an element of $Epi(\Sets^{op})$. Other notable examples are the interpolation categories for finite dimensional complex representations of wreath products or finite linear groups \cite{EAH}.

While it is often easy to see that two given representation categories are related by functors (for example, restriction and induction), the question is more delicate for the Deligne-Knop categories. In Section \ref{sec:functoriality} we adress the question of functoriality. We show in Lemma \ref{lem:knop_constr_functorial} that if $F: \mathcal{A}\to \mathcal{A}'$ is a functor preserving the degree functions and pullbacks on two regular categories $\mathcal{A},\mathcal{A}'$, it extends naturally to a $\K$-linear functor $\overline{F}:\mathcal{T}(\mathcal{A},\delta) \longrightarrow \mathcal{T}(\mathcal{A}',\delta')$ which is symmetric monoidal if $F$ preserves terminal objects.

In order to have enough examples, we study two additional examples of regular categories (not discussed in \cite{K}), the category $G-\Sets^{op}$ for a finite group $G$ and the category $\Rep(G,\F)$, the category of finite dimensional representations of a finite group $G$ over over a finite field $\F$. We discuss the question of functoriality then in a number of examples in Sections \ref{ssec:examples_Knop_constr}, \ref{sec:new-ex}.
\begin{itemize}
\item Functors from \texorpdfstring{$\Vect{\F}$}{vector spaces} to \texorpdfstring{$\Sets^{op}$}.
\item Functors between the categories $G-\Sets_{free}$ and $ \Sets$.
\item Functors $G-\Sets_{free} \to H-\Sets_{free}$ for $H \subset G$.
\item Functors $\Rep(G,\F) \to \Vect{\F}$ for a finite group $G$ and a finite field $\F$.
\end{itemize}

For example, we obtain by functoriality symmetric monoidal functors $\mathcal{T}(\Vect{\F},\delta_t) \cong \underline{\Rep}(GL_t(\F)) \to \mathcal{T}(\Rep(G,\F),\delta_{\tau'_t})$ for certain degree functions $\delta_t$ and $\delta_{\tau'_t}$. The existence of such functors also follows from the universality properties of $\underline{\Rep}(GL_t(\F))$, but in our approach we don't need to know anything about a presentation of the categories or their universal properties. 

In Proposition \ref{prop:adjoint} we give a criterion when a pair $F: \mathcal{A}\rightleftarrows \mathcal{A}':G$, $G \vdash F$ of adjoint functors preserving the degree functions leads to an adjoint pair $\overline{G} \vdash \overline{F}$ between the Deligne-Knop categories. An immediate application is that in this case the induced $\K$-linear symmetric monoidal functor $\overline{F}:\mathcal{T}(\mathcal{A},\delta) \longrightarrow \mathcal{T}(\mathcal{A}',\delta')$ preserves limits.

An example where the criteria are met is given by the pair $F : \mathcal{A}\rightleftarrows \mathcal{A}/x_0:G$ where $\mathcal{A}/x_0$ is the slice category (overcategory) over a fixed object $x_0 \in \mathcal{A}$ and where $F:= (-) \times x_0$ and $G$ is the forgetful functor $(x, p_x)\mapsto x$. In the case $\mathcal{A}:=\Vect{\F_q}$ the category $\mathcal{T}(\Vect{\F_q}/V_0, \delta'_t)$ interpolates representations of a certain maximal parabolic subgroups of $GL_n(\F_q)$ for different values of $n$ as in Remark \ref{remark-max-par}, and the induced functor $\overline{F}: \underline{\Rep}(GL_{t}(\F_q)) \to  \mathcal{T}(\Vect{\F_q}/V_0, \delta'_t)$ preserves finite limits.

The examples in Sections \ref{ssec:examples_Knop_constr}, \ref{sec:new-ex} and \ref{sec:examples_adj_pairs} show, that while there are many adjoint pairs of functors on the level of regular categories, most do not lift to a pair of adjoint functor between the Deligne-Knop categories.

It was proven by Snowden \cite{Snowden} that Knop's categories can all be obtained from the theory of oligomorphic groups \cite{HS}. It would be interesting to study functoriality and adjointness for the monoidal categories arising from oligomorphic groups and their measures. Since the measures can be thought of as the replacement of Knop's degree functions, the restriction that the functors need to preserve the degree functions should also usually prevent the lifting of adjoint pairs to the oligomorphic monoidal categories.

\subsection*{Acknowledgements} The research of I. E. was supported by ISF grant 1362/23.


\section{Preliminaries}
We denote $\mathbb{N}:=\mathbb{Z}_{\geq 0}$.

The base field for this paper will be denoted by $\mathbb{K}$. All the categories and the functors will be $\K$-linear, unless stated otherwise. In many cases, we will also require that $char \, \K=0$.

We will only consider well-powered categories (so that for any object $x$, the its subobjects form a set).

\subsection{Monoidal categories} Let $\K$ be any field. We adopt the same notion of a $\K$-linear symmetric monoidal (SM) category as \cite[Chapters 2 and 8]{EGNO}. In particular we have a binatural family of braiding morphisms $\gamma_{XY}:X\otimes Y\stackrel{\sim}{\to} Y\otimes X$ which satisfy the constraints of \cite[Chapter 8]{EGNO}. 
\medskip
A $\K$-linear functor $F:\mathcal{C}\to \mathcal{C}'$ between two SM categories $\mathcal{C}$ and $\mathcal{C}'$ is called symmetric monoidal ({\it SM} for short) if it is equipped with natural isomorphisms $c_{XY}^F:F(X)\otimes F(Y)\stackrel{\sim}{\to}F(X\otimes Y)$ and $\one\stackrel{\sim}{\to}F(\one)$ satisfying the usual compatibility conditions \cite[Section 8]{EGNO}. 


\section{Knop's construction}
\subsection{Knop-Deligne ``skeletal'' categories}\label{ssec:Knop_category_construct}
We recall a construction of Knop \cite{K}.

The following notion of regularity was introduced in \cite{K} to extend the notion of a complete category. A category $\mathcal{A}$ is called {\it regular} if 
\begin{itemize}
    \item $\mathcal{A}$ has a terminal object $\mathbf{1}_{\mathcal{A}}$, 
    \item all morphisms have images, 
        \item there exists a pullback $x\times_z y$ for every commutative diagram $\xymatrix{&w \ar[r] \ar[d] &y \ar[d] \\ &x\ar[r] &z} $,
    \item for any $x\twoheadrightarrow z$, $y\to z$ there exists a pullback $x\times_z y$,
    \item epimorphisms are preserved under pullbacks. 
\end{itemize}

A  {\it relation} between objects $x$ and $y$ of $\mathcal{A}$ is a subobject $r$ of $x \times y$. Given another relation $s \subset y \times z$, their {\it product} is the relation $r \star s :=  im(r \times_y s \to x\times z) \subset x\times z$, as demonstrated in the diagram below:
    $$\xymatrix{ & & &{r \times_y s} \ar[ld] \ar[rd] \ar@{->>}[d] & & & \\ & &{r} \ar[ld] \ar[rd] &{r \star s} \ar@{-->}[lld] \ar@{-->}[rrd] &{s} \ar[ld] \ar[rd] & \\ &x & &y & &z}$$

The regularity of the category $\mathcal{A}$ guarantees that this product is associative. 

Let $Epi(\mathcal{A})$ be the set of all epimorphisms in $\mathcal{A}$. A {\it degree function} on $\mathcal{A}$ is a map $\delta:Epi(\mathcal{A} ) \to \K$, subject to the following conditions (see \cite[Definition 3.1]{K}):
\begin{itemize}
    \item Multiplicativity: $\delta(e'e) = \delta(e')\delta(e)$ for any two epimorphisms $e,e'\in Epi(\mathcal{A})$ which can be composed; additionally, $\delta(\id_x) = 1$ for all $x\in \mathcal{A}$.
    \item $\delta$ preserves pullbacks: $\delta(e') = \delta(e)$ whenever $e'$ is a pull-back of $e$.
\end{itemize}

Given a regular category $\mathcal{A}$ and a degree function, Knop defines the skeletal category $\mathcal{T}^0(\mathcal{A},\delta)$ as follows \cite[Definition 3.2]{K}:
\begin{itemize}
\item The objects are the objects of $\mathcal{A}$. Following Knop the object $x \in \mathcal{A}$ is denoted by $[x]$ when seen as an object in $\mathcal{T}^0(\mathcal{A},\delta)$.
\item The morphisms $[x] \to [y]$ are the formal $\K$-linear combinations of relations $r \subset x \times y$. We denote the morphism corresponding to $r$ by $\rel{r}:[x] \to [y]$.
\item If $r \subset x \times y$ and $s \subset y \times z$ are two relations, then we set $\rel{s}\circ \rel{r} := \delta(e) \rel{s\star r}$ if $r \times_y s$ exists and $e:r \times_y s \twoheadrightarrow s\star r$, while $\rel{s}\circ \rel{r} :=0$ otherwise.
\end{itemize}

\begin{remark}
A special case of this construction is the {\it category of relations} $Rel(\mathcal{A})$, which has the same objects as $\mathcal{A}$ but the morphism spaces are given by the $\K$-span of all relations between $x, y$. The composition in this category is given by the product of relations, and is the category $\mathcal{T}^0(\mathcal{A},\delta)$ when $\delta\equiv 1$ is the constant function.
\end{remark}

\subsection{Embedding of the original category}\label{ssec:embedding_orig_cat}
The category $\mathcal{T}^0(\mathcal{A},\delta)$ comes equipped with a contravariant isomorphism
    \begin{align*} 
    (-)^{\vee}:  \mathcal{T}^0(\mathcal{A},\delta)&\longrightarrow \mathcal{T}^0(\mathcal{A},\delta)^{op}\\
    [x]&\longmapsto [x]\\
    \left(\rel{r}:[x]\to [y] \right) \;\; &\longmapsto   \;\;   \left( \rel{r}: [y]\to [x] \right) 
       \end{align*}
    and the natural embedding
   \begin{align*}
       \Gamma_{\mathcal{A}}: \mathcal{A}&\longrightarrow \mathcal{T}^0(\mathcal{A},\delta)\\
       x \;&\longmapsto \;[x]\\
         \left(f:x\to y\right)  \;\;&\longmapsto  \;\; \left(\rel{\Gamma f}: [x] \to [y]\right)
   \end{align*}
where $\Gamma f\subseteq x\times y$ denotes the graph of $f:x\to y$.

It is proved in \cite[Theorem 3.1]{K3} that the morphisms in $\mathcal{T}^0(\mathcal{A},\delta)$ are generated by morphisms of the form $\rel{\Gamma f}, \rel{\Gamma f}^{\vee}$ under the relations
\begin{itemize}
    \item $\rel{\Gamma f} \circ \rel{\Gamma g } = \rel{\Gamma {fg}}$, $\rel{\Gamma g}^{\vee} \circ \rel{\Gamma f }^{\vee} = \rel{\Gamma {fg}}^{\vee}  $ for any two composable morphisms $f, g$ in $\mathcal{A}$.
    \item $\rel{\Gamma f}^{\vee}\circ \rel{\Gamma g} = \rel{\Gamma {g'}}\circ \rel{\Gamma {f'}}^{\vee}$, for any Cartesian square $  \xymatrix{&r \ar_{f'}[d]  \ar^{g'}[r] & y \ar^{f}[d] \\&x \ar^{g}[r]  &z 
}$ in $\mathcal{A}$.
    \item $\rel{\Gamma e} \circ \rel{\Gamma e}^{\vee} = \delta(e)\id_{[y]}$ for any $e:x\twoheadrightarrow y$ in $Epi(\mathcal{A})$.
\end{itemize}

\begin{remark}
In fact, the equality $\rel{\Gamma f}^{\vee}\circ \rel{\Gamma g} = \rel{\Gamma {g'}}\circ \rel{\Gamma {f'}}^{\vee}$ rarely holds if $  \xymatrix{&r \ar_{f'}[d]  \ar^{g'}[r] & y \ar^{f}[d] \\&x \ar^{g}[r]  &z 
}$ is not a Cartesian square. Indeed, by the definition of composition in $\mathcal{T}^0(\mathcal{A},\delta)$, the above equality of morphisms holds if and only if the canonical map $r\to x\times_z y$ is an epimorphism and $\delta(r\twoheadrightarrow x\times_z y)=1$. 
\end{remark}
 
\subsection{Karoubi envelope}
The category $\mathcal{T}(\mathcal{A},\delta)$ is defined as the additive Karoubi completion of $\mathcal{T}^0(\mathcal{A},\delta)$. The rule $[x] \otimes [y] := [x \times y]$ induces a symmetric monoidal structure on $\mathcal{T}(\mathcal{A},\delta)$. 

\subsection{Examples (due to F. Knop)}\label{ssec:examples_Knop_constr}

Recall from \cite[Section 3.2]{K} that there are no non-trivial degree functions on the category of finite sets $\Sets$, but there are nice degree functions on the categories $\Sets^{op}$, $\Vect{\F}$ as well as the category $G-\Sets^{op}$ (respectively, $G-\Sets^{op}_{free}$), which is the category of finite sets with an action of a finite group $G$ (respectively, a free action). 

As shown in \cite{K}, these categories are regular, and the following holds:
\begin{enumerate}
\item The category $\Sets^{op}$ is the category of finite boolean algebras. The degree functions on this category are parameterized by $t\in \K$, with $\delta_t(e):=t^{\abs{Y\setminus e(X)}}$ where $e: X\hookrightarrow Y$ in $\Sets$ is viewed as an element of $Epi(\Sets^{op})$. Knop's construction gives rise to the family of Deligne categories $\underline{\Rep}(S_t):=\mathcal{T}(\Sets^{op}, \delta_t)$, $t \in \K$. In the case of $char \, \K =0$, the categories $\underline{\Rep}(S_t)$ are semisimple if and only if $t\notin \mathbb{N}$. This family interpolates the categories of finite-dimensional representations of the symmetric groups.

\item For a fixed finite field $\F$, consider the category $ \Vect{\F}$ of finite-dimensional vectors spaces over $\F$. The degree functions on $ \Vect{\F}$ are also parameterized by $t\in \K$, with 
$\delta_t(e):=t^{\dim \ker e}$ for any $e\in Epi(\Vect{\F})$.

In the case of $char \, \K =0$ and $\F = \F_q$, the categories $\underline{\Rep}(GL_t(\F_q)):=\mathcal{T}(\Vect{\F_q}, \delta_t)$ are semisimple if and only if $t\notin |\F|^\mathbb{N}$. The family of categories $\underline{\Rep}(GL_t(\F_q))$, $t\in \K$  interpolates the categories of finite-dimensional representations of $GL_n(\F_q)$. 

\item The category $G-\Sets_{free}^{op}$ is the category of finite boolean algebras with a locally free $G$-action. 
The degree functions on this category are parameterized by $t\in \K$. We denote by $\delta_t$ the degree function corresponding to $t\in \C$, with $$\delta_t(e) := t^{|Y\setminus e(X)|/G}$$ where $e\in Epi(G-\Sets_{free}^{op})$ is given by an injective map $e:X\to Y$ in $G-\Sets_{free}^{op}$, and the value $|(-)/G|$ denotes the number of $G$-orbits in a given finite set.

We denote $\underline{\Rep}(G \wr S_t):=\mathcal{T}(G-\Sets_{free}^{op}, \delta_t)$. In the case of $char \, \K =0$, the categories $\underline{\Rep}(G \wr S_t)$ are semisimple if and only if $t\notin \mathbb{N}|G|$.
The family of categories $\underline{\Rep}(G \wr S_t)$, $ t\in \K$ interpolates the categories of finite-dimensional representations of the wreath product $G \wr S_n$. For example, if $G = \mathbb{Z}/2\mathbb{Z}$, these categories interpolate the representation
categories of the hyperoctahedral groups, i.e., the Weyl groups of type $BC_n$.

\end{enumerate}
\subsection{New examples}\label{sec:new-ex}
The following two categories have not been discussed in \cite{K}, so we work out the details here. 
\subsubsection{\texorpdfstring{$G-\Sets^{op}$}{G-Sets opposite}}\label{sssec:iterated-wreath_cat_constr} 
Let $G$ be a finite group. Consider the category $G-\Sets$ of all finite $G$-sets; it is regular, and so the opposite category is regular by \cite[Section 2.3]{K}.

Consider the following two posets: 
\begin{itemize}
    \item The poset $Tran(G)$ of non-empty transitive $G$-sets, up to $G$-isomorphisms, with a partial order given by $G$-maps: if $S, S'\in Tran(G)$ and there is a $G$-map $S\to S'$  (such a map is automatically surjective) then $S\geq S'$.
    \item The poset $SGp_G$ of subgroups of $G$, with the partial order given by inclusion.
\end{itemize} 
Then the opposite poset to $Tran(G)$ is the poset $SGp_G/G$ of conjugacy classes of subgroups of $G$ with the order inherited from $SGp_G$. This correspondence is given by
$$H\in SGrp_G \; \longmapsto\; G/H, \;\;\;\; X \in Tran(G) \; \longmapsto\; G_x \text{ for arbitrary } x\in X,$$ where $G_x$ denotes the stabilizer of $x\in X$ (the conjugacy class of $G_x$ does not depend on the choice of $x$).
For $X\in Tran(G)$, we will denote by $G_X:=G_x$ the element of $SGp_G/G$ corresponding to $X$. Clearly, we have an isomorphism of $G$-sets $G/H\cong G/H'$ if and only if $H, H'\in SGrp_G$ are conjugate, and for any $X\in Tran(G)$, we have: $G/G_{X}\cong X$ as $G$-sets.

The degree functions on $G-\Sets^{op}$ are parameterized by the set of functions $Tran(G)\to \K$. Given a function 
$\phi:Tran(G)\to \K$, the corresponding degree function $\delta_\phi $ is given by 
$$\delta_\phi(e):=\prod_{S\in (Y\setminus e(X))/G} \phi(S)$$
where the $G$-equivariant monomorphism $e:X\hookrightarrow Y$ is considered as an epimorphism in $G-\Sets^{op}$ and $S$ is a $G$-orbit in $Y\setminus X$.
\begin{remark}
Note that for $e\in Epi(G-\Sets_{free}^{op})$, we have: $\delta_{\phi}(e)=\delta_{t'}(e)$ where $t':=\phi(\{1_G\})$ and $\delta_{t'}$ is the degree function on $G-\Sets_{free}^{op}$ defined in \cref{sssec:wreath_G_sets_ex}.
\end{remark}

\emph{Semisimplicity of the category $\mathcal{T}(G-\Sets^{op},\delta_{\phi})$:}

Let $\mu$ be the Mobius function on the poset $Tran(G)$.
By \cite[Theorem 8.6]{K}, the category $\mathcal{T}(G-\Sets^{op},\delta_{\phi})$ is semisimple if and only if we have: $$\forall \;  S\in Tran(G), \;\;\; \;\sum_{\substack{S'\in Tran(G)\\ S'\leq S}} \mu(S, S') \phi(S') \notin \mathbb{N}.$$

A necessary (though not sufficient for $G\neq \{1\}$) condition for this to hold is that $\phi(\{\bullet\})\notin \mathbb{N}$. 
\newline

\emph{Interpolation property of the family $\mathcal{T}(G-\Sets^{op},\delta_{\phi})$:} 

We now determine the interpolation property. Following \cite[Theorem 9.8]{K}, the family $\mathcal{T}(G-\Sets^{op},\delta_{\phi})$ interpolates the categories of continuous finite dimensional representations of profinite groups $\Aut(P)$ over $\K$. Here $\Aut(P)$ is the group of automorphisms of a uniform functor $P: G-\Sets^{op} \to \Sets$ in the sense of \cite[Definition 9.2]{K}.

Applied to $ G-\Sets^{op}$, the uniform functors are of the form \[ P_X(A): = \Hom_{G-\Sets} (A,X)  \] for some fixed finite $G$-set $X$. In the terminology of \cite[Section 9.3]{K}, the functor $P_X$ has a unique degree function on $G-\Sets^{op}$ adapted to it: that's the function $\delta_{\phi}$ corresponding to $$\phi: Tran(G) \to \K, \;\; \phi(S)\,:=\, |\{S'\in X/G \, : \, S\geq S'\}| $$
(so $\phi(S)$ is the number of $G$-orbits in $X$ onto which $S$ has a $G$-map). We therefore need to determine the automorphism group of the functor $P_X$ for fixed $X \in G-\Sets^{op}$. This automorphism group agrees with $\Aut_G(X)$, which we will now compute.

Let $X$ be a $G$-set. If $X$ is transitive, and $x \in X$ is arbitrary, \[ \Aut_G(X) \cong N_{G}(G_{x})/G_x \cong N_G(G_X)/G_X\] by \cite[Theorem 11.28]{Lee} (we remind the reader that $G_X$ denotes the $G$-conjugacy class of the stabilizer $G_x$ of an arbitrary element $x\in X$).
Hence the automorphism group of a $G$-set $X$ contains the subgroup \[ \prod_{S \in X/G} N_G(G_S)/G_S \] together with an additional permutation action between isomorphic $G$-orbits. Using the identification between $Tran(G)$ and $SGp_G/G$, we obtain \[ \Aut (X) = \prod_{H \in SGp_G/G} S_{n_H} \wr (N_G(H)/H)\] where the product runs over the conjugacy classes of subgroups of $G$ and $n_H$ is the number of orbits $S\in X/G$ such that $H=G_x$ for some $x\in S$ (in other words, $H$ is an element of the conjugacy class $G_S$). Then \cite[Theorem 9.8]{K} implies the following about the semisimplification of $\mathcal{T}(G-\Sets^{op},\delta_{\phi})$. 

\begin{corollary} 
Let $char \K=0$ and let $n:SGp_G/G \to \mathbb{N}$ be a function. Consider the function $\phi:SGp_G/G \to \mathbb{N}$ given by $$\phi(H):=\sum_{\substack{K\in SGp_G/G,\\ H\leq K}} n(H)$$Identifying $\phi$ with the corresponding function $\phi: Tran(G) \to \mathbb{N}\subset\K$, we obtain: the semisimplification of the symmetric monoidal category $\mathcal{T}(G-\Sets^{op},\delta_{\phi})$ is \[\Rep\left( \prod_{H \in SGp_G/G} S_{n(H)} \wr \,(N_G(H)/H), \; \K\right).\]

\end{corollary}

Hence the family of categories $\mathcal{T}(G-\Sets^{op},\delta_{\phi})$ interpolates the categories of representations of wreath products $\prod_{H \in SGp_G/G} S_{n(H)} \wr (N_G(H)/H)$.

 \subsubsection{Representations of a finite group over a finite field}\label{sssec:rep_finite_group}
Let $\F$ be a finite field.
Let $G$ be a finite group and consider $\Rep(G,\F)$, the category of finite dimensional representations of $G$ over $\F$. Then $\Rep(G,\F)$ is a regular category.

Since $\Rep(G,\F)$ is an abelian category, its degree functions can be determined from \cite[Example 4, Section 3]{K}. Let $Irr(G)$ denote the set of irreducible representations of $G$. For $V \in \Rep(G,\F)$ and $I \in Irr(G)$ denote by $[V:I]$ the multiplicity of $I$ in $V$. The degree functions on $ \Rep(G,\F)$ are in bijection with functions $\tau: Irr(G) \to \mathbb{K}$. The degree function corresponding to $\tau$ is 
\[ \delta_{\tau}(e:V \twoheadrightarrow W) = \prod_{I \in Irr(G)} \tau(I)^{[\ker(e):I]} .\] 

Given $\tau: Irr(G) \to \mathbb{K}$, we denote $t_{\tau}:=\prod_{I\in Irr(G)} \tau(I)^{\dim I}$, a notation we will use later on.

\begin{example} 
Let $t\in \K$.
\begin{enumerate}
    \item Define the function $\tau_t:Irr(G) \to \K$, $\tau_t(I):=t^{\dim I}$ as before. Then $t_{\tau_t}=t^{|G|}$.
    \item We may define a function $\tau'_t:Irr(G) \to \K$ by letting $$\tau'_t(I)=\begin{cases}
    1 &\text{ if } I\neq \one\\
    t &\text{ if } I=\one
\end{cases}.$$ Then $t_{\tau'_t}=t$. Clearly, $\tau'_t=\tau_t$ if and only if $G=\{1\}$.
\end{enumerate}
\end{example}

From now on, assume that $char \,\F\nmid |G| $.

\emph{Semisimplicity of the category $\mathcal{T}(\Rep(G,\F),\delta_{\tau})$:}

Using \cite[Section 8, Example 5]{K} we see that $\mathcal{T}(\Rep(G,\F),\delta_{\tau})$ is semisimple if and only if $\tau(I) \notin |\End_G(I)|^{\mathbb{N}}$ for all $I  \in Irr(G)$. Now, the division algebra $\F_I:=\End_G(I)$ is finite dimensional over $\F$ and therefore finite. Since finite division algebras are fields, $\F_I$ is in fact a field extension of $\F$ and hence $|\F_I| = |\F|^n$ for some $n\in \mathbb{N}$. We conclude that $\mathcal{T}(\Rep(G,\F),\delta_{\tau})$ is semisimple if and only if $\tau(I)\notin |\F_I|^{\mathbb{N}}$ for all $I\in Irr(G)$.

In particular, $\mathcal{T}(\Rep(G,\F),\delta_{\tau_t})$ is semisimple whenever $t\notin \mathbb{N}$, while $\mathcal{T}(\Rep(G,\F),\delta_{\tau'_t})$ is not semisimple for any values of $t$ (since $0 \in \mathbb{N}$). 
\newline

\emph{Interpolation property of the family $\mathcal{T}(\Rep(G,\F),\delta_{\tau})$:}

As in \cite[Theorem 9.8]{K}, the family $\mathcal{T}(\Rep(G,\F),\delta_{\tau})$ interpolates the categories of continuous finite dimensional representations of profinite groups $\Aut(P)$ over $\K$. Again, $\Aut(P)$ is the group of automorphisms of a uniform functor $P: \Rep(G,\F) \to \Sets$ which is compatible with a degree function $\delta_{\tau}$, in the sense of \cite[Definition 9.2]{K}.

The uniform functors $P: \Rep(G,\F) \to \Sets$ are of the form $P = P_V = \Hom_{G}(V,-)$ for $V \in \Rep(G,\F)$. The uniform functor $P$ is compatible with the degree function $\delta_{\tau}$, where $\tau:Irr(G)\to \K$ is given by $\tau(I):=|\F_I|^{[V:I]}$.

Now, the automorphism group of the functor $P$ is just the group of automorphisms of the $G$-representation $V$: $$\Aut(P) = \Aut_G(V) \cong \prod_{I\in Irr(G)} GL_{[V:I]}(\F_I)$$ where $\F_I:=\End_G(I)$ is a finite field extension of $\F$, as before, and $GL_{[V:I]}(\F_I)$ is the general linear group of rank $[V:I]$ over this field.

\begin{remark}
Without the assumption that $char \,\F\nmid |G| $, one would need to restrict ourselves only to injective $G$-modules $V$.
\end{remark}

\begin{corollary}
Let $char \K=0$ and let $n:Irr(G) \to \mathbb{N}$ be a function. Define a function $\tau:Irr(G)\to \mathbb{N}\subset \K$ by $\tau(I):=|\F_I|^{n(I)}$. Then the semisimplification of the symmetric monoidal category $\mathcal{T}(\Rep(G,\F),\delta_{\tau})$ is \[\Rep\left( \prod_{I\in Irr(G)} GL_{n(I)}(\F_I), \; \K\right).\]

\end{corollary}

Hence the family of categories $\mathcal{T}(\Rep(G,\F),\delta_{\tau})$ interpolates the categories of representations of products $\prod_{I\in Irr(G)} GL_{n(I)}(\F_I)$.

\section{Functoriality of Knop's construction}\label{sec:functoriality}
\subsection{Functoriality}
\begin{definition}\label{def:functor_preserving_degree}
   Let $\mathcal{A},\mathcal{A}'$ be two regular categories and let $\delta, \delta'$ be degree functions on $\mathcal{A},\mathcal{A}'$ respectively. Let $F: \mathcal{A}\to \mathcal{A}'$ be a functor. We will say that $F$ {\it preserves the degree functions} if $F$ preserves epimorphisms and $\delta'(F(e))=\delta(e)$ for every epimorphism $e$ in $\mathcal{A}$.
\end{definition}

\begin{lemma}\label{lem:knop_constr_functorial}
   Let $\mathcal{A},\mathcal{A}'$ be two regular categories and let $F: \mathcal{A}\to \mathcal{A}'$ be a functor preserving the degree functions and pullbacks. Then the following statements hold:

\begin{enumerate}
    \item There exists a functor
$\overline{F}:\mathcal{T}^0(\mathcal{A},\delta) \longrightarrow \mathcal{T}^0(\mathcal{A}',\delta')$ defined as follows:
\begin{itemize}
    \item On objects: for $x\in \mathcal{A}$, we put $\overline{F}([x]):=[F(x)]$.
        \item On morphisms: for $x,y\in \mathcal{A}$ and $r\subset x\times y$, we define $\overline{F}(\rel{r}):=\rel{F(r)}$
        where $F(r)\subset F(x\times y)$ is viewed as a subobject of $F(x)\times F(y)$ via the monomorphism $F(x\times y)\to  F(x)\times F(y)$.
\end{itemize}
\item We have a natural isomorphism $\Gamma_{\mathcal{A}'} \circ F \cong \overline{F}\circ \Gamma_{\mathcal{A}} $.

\item Assume $F$ preserves the terminal objects (that is, if $F(\mathbf{1}')=F(\mathbf{1})$; equivalently, if $F$ preserves all finite limits which exist in $\mathcal{A}$). Then the functor $\overline{F}$ is symmetric monoidal.
\end{enumerate}
\end{lemma} 
The functor $\overline{F}$ extends naturally to a $\K$-linear functor $\overline{F}:\mathcal{T}(\mathcal{A},\delta) \longrightarrow \mathcal{T}(\mathcal{A}',\delta')$ which is symmetric monoidal if $F$ preserves terminal objects.

\begin{remark}
The fact that a functor $F$ preserves pullbacks doesn't imply that it preserves products: $F$ doesn't necessarily take the terminal object in $\mathcal{A}$ to the terminal object in $\mathcal{A}'$ (see for example the functor $G$ described in \cref{ssec:slice_cat}). However, for any $x, y\in  \mathcal{A}$, we have: the natural map $F(x\times y)  \to F(x) \times F(y)$ is a monomorphism.

Indeed, for any two arrows $x' \to z', y'\to z'$  in $\mathcal{A}'$, the natural map $x'\times_{z'} y'\to x'\times y'$ is always a monomorphism. If the functor $F$ preserves pullbacks, then for any object $x, y\in  \mathcal{A}$, the map $$F(x\times y) = F(x\times_{\mathbf{1}'} y) = F(x)\times_{F(\mathbf{1})} F(y) \to F(x) \times F(y)$$ is a monomorphism.
\end{remark}

\begin{proof}[Proof of \cref{lem:knop_constr_functorial}]
    For the first statement we need to check that given $r\subset x\times y$, $s\subset y\times z$ in $\mathcal{A}$, we will have: $F(r \star s)=F(r)\star F(s)$. Recall that $F$ preserves epimorphisms and pullbacks (in particular, it preserves monomorphisms). So $F$ preserves images, too. Thus we have: \begin{align*}
        F(r \star s)&=  F(im(r \times_y s \to x\times z)) = im(F(r \times_y s \to x\times z)) =\\ &im(F(r) \times_{F(y)} F(s) \to F(x)\times F(z)))= F(r)\star F(s) 
    \end{align*} where the third equality is due to the fact that $F$ preserves pullbacks and epimorphisms. 

    The second statement of the lemma follows directly from the fact that $F$ preserves pullbacks and thus $\Gamma {F(f)}=F(\Gamma f)$ for any $f:x\to y$ in $\mathcal{A}$.

    The third statement is straightforward.
\end{proof}

\subsection{Examples} \label{ssec:examples_functoriality}

\subsubsection{Functor from \texorpdfstring{$\Vect{\F}$}{vector spaces} to \texorpdfstring{$\Sets^{op}$}{Sets}}\label{sssec:forgetful_functor_ex}
Let $\F$ be a finite field.

Consider the functor $\F[-]:  \Sets^{op} \to \Vect{\F}$ sending a finite set $X$ to the space of $\F$-valued functions
on $X$. Its left adjoint is the functor $(-)^*\circ Forget: \Vect{\F} \to \Sets^{op}$, sending a vector space $V$ to the underlying space of its dual $V^*$.

Clearly, $(-)^*\circ Forget$ doesn't preserve pullbacks, while the functor $\F[-]$ preserves pullbacks and terminal objects. 

Let $t\in \K$. We denote by $\delta_t$, $\delta'_t$ the corresponding degree functions on the categories $\Sets^{op}$, $ \Vect{\F}$ respectively.
The following is a direct corollary of \cref{lem:knop_constr_functorial}: 
\begin{lemma} The functor $\F[-]$ preserves the degree functions, the pullbacks and the terminal objects. In particular, we obtain a symmetric monoidal functor $$\overline{\F[-]}: \mathcal{T}(\Sets^{op},\delta_t) \to \mathcal{T}(\Vect{\F},\delta'_t)$$ or, in other words, a symmetric monoidal functor $\overline{\F[-]}: \underline{\Rep}(S_t) \to \underline{\Rep}(GL_t(\F))$.
\end{lemma}
This functor is described by the universal property of $\underline{\Rep}(S_t)$ (see \cite{Del07}), stating that each such symmetric monoidal functor out of $\underline{\Rep}(S_t)$ corresponds to a special commutative Frobenius object in the target category. In this case, the special commutative Frobenius object is given by the generating object of $\underline{\Rep}(GL_t(\F))$: the tautological ``matrix representation'' of $GL_t(\F)$ (see \cite{EAH} for a description of this object as a $\mathbb{F}_q$-linear Frobenius space).

\subsubsection{Wreath products}\label{sssec:wreath_G_sets_ex}
Let $G$ be a finite group. There is a forgetful functor $\Res: G-\Sets_{free} \to \Sets$ which forgets the $G$-action. This functor is representable: it is represented by the $G$-set $G$ on which the group acts by left multiplication. We obtain a pair of adjoint functors $(G\times -)\vdash \Res$, where $$ G\times -: \Sets \to G-\Sets_{free}, \;\;\text{ and }\;\;\Res: G-\Sets_{free} \to \Sets.$$
\begin{remark}
In fact, this is part of the triple of adjoint functors $(G\times -)\vdash \Res \vdash Fun(G, -)$ between the categories $G-\Sets$ and $ \Sets$. They define a triple of adjoint functors $$Fun(G, -)\vdash \Res \vdash (G\times -)$$ between $G-\Sets^{op}$ and $\Sets^{op}$. Hence the functors $\Res\,:\, G-\Sets^{op} \leftrightarrows \Sets^{op}\,:\, (G\times -)$ preserve pullbacks.
\end{remark}
\comment{ 



}

The functor $G\times -$ yields a functor $\Ind: \Sets^{op} \to G-\Sets^{op}_{free}$. Choose $t\in \K$ and consider the corresponding degree functions $\delta_t, \delta'_t$ on $\Sets^{op}, G-\Sets^{op}_{free}$, respectively (see \cref{ssec:examples_Knop_constr}). The functor $\Ind$ preserves the degree functions $\delta_t$ and $\delta'_t$, as well as the terminal objects (the empty set) and the pullbacks. 

The forgetful functor $\Res:G-\Sets^{op}_{free} \to \Sets^{op}$ is the left adjoint of $\Ind$ and preserves pullbacks and terminal objects, yet it doesn't preserve the degree functions $\delta_t, \delta'_t$. Indeed, for $e: Y\hookrightarrow Y'$ a monomorphism in $G-\Sets_{free}$, seen as an element of $Epi(G-\Sets_{free}^{op})$, we have
$$\delta_t(\Res(e))=t^{|Y'\setminus e(Y)|}\neq t^{|(Y'\setminus e(Y))/G|}=\delta'_t(e).$$ However, this can be remedied by considering the degree functions $\delta_t$ on $\Sets^{op} $ and $\delta'_{t^{|G|}}$ on $G-\Sets^{op}_{free}$.

The following lemma is an immediate corollary of \cref{lem:knop_constr_functorial}.

\begin{lemma} \label{lem:wreath} 
\mbox{}
\begin{enumerate}
    \item 
The functor $\Ind$ preserves the degree functions $\delta_t, \delta'_t$, the pullbacks and the terminal objects.  In particular, we obtain a symmetric monoidal functor $$\overline{\Ind}: \mathcal{T}(\Sets^{op},\delta_t) \to \mathcal{T}(G-\Sets^{op}_{free},\delta'_t)$$ or, in other words, a symmetric monoidal functor $\overline{\Ind}: \underline{\Rep}(S_t) \to \underline{\Rep}(G \wr S_t)$.
\item The functor $\Res$ preserves the degree functions $\delta_t$ and $\delta'_{t^{|G|}}$, the pullbacks and the terminal objects.  In particular, we obtain a symmetric monoidal functor $$\overline{\Res}: \mathcal{T}(G-\Sets^{op}_{free},\delta'_{t^{|G|}})\to \mathcal{T}(\Sets^{op},\delta_t) $$ or, in other words, a symmetric monoidal functor $\overline{\Res}: \underline{\Rep}(G \wr S_{t^{|G|}}) \to \underline{\Rep}(S_t)$.
\end{enumerate}
\end{lemma}

In terms of the universal property of $\underline{\Rep}(S_t)$, the functor $\overline{\Ind}$ corresponds to the generating object $V$ (the categorical analogue of the permutation representation) of $\underline{\Rep}(G \wr S_t)$. By \cite[Remark 4.2]{LS}, this object carries the structure of a special commutative Frobenius object.

\subsubsection{Relative wreath products}\label{sec:relative-wreath}

More generally consider a subgroup $H \subset G$ of a finite group $G$. Then we have a restriction (forgetful) functor \[ Res_H^G: G-\Sets \;\longrightarrow \; H-\Sets. \] It admits a left and a right adjoint \[ Ind_H^G, Coind_H^G : H-\Sets \;\longrightarrow \; G-\Sets, \] defined as follows. Let $Y\in H-\Sets$. Then we take $Coind_H^G(Y) := \Hom_H(G, Y)$ and $Ind_H^G(Y) := G \times Y/\sim$ where $\sim$ is the equivalence relation given by $(gh,y) \sim (g, hy)$ for $g,h \in G, \ y \in Y$. The group $G$ acts on both these $G$-sets via left multiplication on the first component $G$ (and trivial action on $Y$). Thus we have a triple of adjoint functors $Ind_H^G\vdash Res_H^G\vdash Coind_H^G $. The functors $Ind_H^G, Res_H^G$ have right adjoints, so they preserve pushforwards.

Clearly, the functor $Res_H^G$ takes free $G$-sets to free $H$-sets. Furthermore, $Ind_H^G$ takes free $H$-sets to free $G$-sets: since both the regular action of $G$ on itself and the action of $H$ on $Y$ are free, so is the induced action. As in \cref{sssec:wreath_G_sets_ex}, the induction and restriction functors yield an adjoint pair $(\Res, \Ind)$ where $$\Ind: H-\Sets^{op}_{free} \;\leftrightarrows \; G-\Sets^{op}_{free}: \Res.$$

These functors $\Ind, \Res$ preserve  the terminal objects (empty sets) and pullbacks, since $Ind_H^G, Res_H^G$ preserve pushforwards.

For $t\in \K$, the functor $\Ind$ sends a transitive $H$-set to a transitive $G$-set, so it preserves the degree functions $\delta^H_t$ and $\delta^G_t$ on $H-\Sets^{op}_{free} $, $ G-\Sets^{op}_{free}$ respectively (see \cref{ssec:examples_Knop_constr} for definition). The restriction functor $\Res:G-\Sets^{op}_{free} \to H-\Sets^{op}_{free}$ matches the degree function $\delta_t$ on $H-\Sets^{op}_{free}$ with $\delta'_{t^{|G/H|}}$ on $G-\Sets_{free}^{op}$. Hence we obtain, as in \cref{lem:wreath}, the existence of symmetric monoidal functors

\begin{align*} \overline{\Ind}: \underline{\Rep}(H \wr S_t) & \to \underline{\Rep}(G \wr S_t), \\
\overline{\Res}: \underline{\Rep}(G \wr S_t) & \to \underline{\Rep}(H \wr S_{t^{|G/H|}}).\end{align*}

\comment{
\Inna{What about the case when $t^{|G/H|}=t$? Are these adjoint?}
}
\subsubsection{Relative iterated wreath products}

The same reasoning yields an adjoint pair $(\Res, \Ind)$ where $$\Ind: H-\Sets^{op} \;\leftrightarrows \; G-\Sets^{op}: \Res.$$
Again, these functors $\Ind, \Res$ preserve  the terminal objects (empty sets) and pullbacks, since $Ind_H^G, Res_H^G$ preserve pushforwards.

Recall that the degree functions on $ G-\Sets^{op} $ are parameterized by the functions $\phi: Tran(G) \to \K$ on the set of $G$-transitive sets up to $G$-isomorphism.

Given a function $\phi_H: Tran(H)\to \K$, let $\phi_G:Tran(G)\to \K$ be the function given by $\phi_G(S):=\prod_{S'\in Tran(H)} [Res(S):S']$ (here $[Res(S):S']$ denotes the number of $H$-orbits in $S$ which are isomorphic to $S'$). The restriction functor $\Res:G-\Sets^{op} \to H-\Sets^{op}$ matches the degree function $\delta_{\phi_H}$ on $H-\Sets^{op}$ with $\delta'_{\phi_G}$ on $G-\Sets_{free}$.

The functor $\Ind$ sends a transitive $H$-set to a transitive $G$-set, so we obtain a map $Ind:Tran(H) \to Tran(G)$. Given a function $\phi: Tran(G)\to \K$, the function $\phi \circ Ind: Tran(H)\to \K$ defines a degree function on $H-\Sets^{op}_{free} $. Then $\Ind$ matches the degree function $\delta_{\phi  \circ Ind}$ on $ H-\Sets^{op}_{free}$ to the degree function $\delta_{\phi}$ on $ G-\Sets^{op}_{free}$. So for any $\phi_H: Tran(H)\to \K$ and any $\phi: Tran(G)\to \K$, we obtain symmetric monoidal functors

\begin{align*} \overline{\Ind}: \mathcal{T}(H-\Sets^{op},\delta_{\phi}) & \to \mathcal{T}(G-\Sets^{op},\delta_{\phi\circ Ind}), \\
\overline{\Res}: \mathcal{T}(G-\Sets^{op},\delta_{\phi_G}) & \to \mathcal{T}(H-\Sets^{op},\delta_{\phi_H}).\end{align*}

\subsubsection{Products of general linear groups over $\mathbb{F}$}
Let $\F$ be a finite field and $G$ be a finite group. We have the restriction (forgetful) functor \[ \Res:   \Rep(G,\F) \, \longrightarrow \,  \Vect{\F}\] and the induction functor \[ \:\Vect{\F} \, \longrightarrow\,  \Rep(G,\F) , \;\;\; V\, \longmapsto \,\F[G]\otimes V \]  which are both left- and right-adjoint (up to a natural isomorphism) to one another by Frobenius reciprocity. This implies that both functors preserve pullbacks, and clearly they preserve the terminal objects as well.

Let $t\in \K$; it corresponds to the degree function $\delta_t$ on $\Vect{\F}$, as described before. 
Recall the notation from \cref{sssec:rep_finite_group}: for any $t\in \K$ we define $\tau_t: Irr(G) \to \K$ given by $\tau_t(I):=t^{\dim I}$ for each $I\in Irr(G)$; for any  function $\tau: Irr(G) \to \mathbb{K}$ we denote $t_{\tau}:=\prod_{I\in Irr(G)} \tau(I)^{\dim I} \in \K$.

\begin{corollary}\label{cor:ind_res_functoriality_rep_G}

\begin{enumerate}
    \item Fix $t\in \K$. The functor $\Res:  \Rep(G,\F) \to \Vect{\F}$ preserves the degree functions $\delta_t, \delta_{\tau_t}$, the pullbacks and the terminal
objects. In particular, we obtain a symmetric monoidal functor \[ \overline{\Res}: \mathcal{T}(\Rep(G,\F),\delta_{\tau_t})\to \underline{\Rep}(GL_t(\F)):=\mathcal{T}(\Vect{\F},\delta_t) . \]
    \item Fix $\tau: Irr(G) \in \mathbb{K}$. The functor $\Ind: \Vect{\F} \to \Rep(G,\F)$ preserves the degree functions $\delta_{t_{\tau}}, \delta_{\tau}$, the pullbacks and the terminal
objects. In particular, we obtain a symmetric monoidal functor \[ \overline{\Ind}: \underline{\Rep}(GL_{t_{\tau}}(\F)):=\mathcal{T}(\Vect{\F},\delta_{t_{\tau}}) \to \mathcal{T}(\Rep(G,\F),\delta_{\tau}). \]
\end{enumerate}
\end{corollary}
\begin{proof}

The functor $\Res: \Rep(G,\F) \to \Vect{\F}$ preserves the degree functions $\delta_t, \delta_{\tau_t}$: for a given epimorphism $e\in Epi(\Rep(G,\F))$ we have \[ \delta_t (\Res(e)) = t^{\dim \ker(e)} = \prod_{I \in Irr(G)} t^{[\ker(e):I]\dim I} = \delta_{\tau_t}(e). \]

Now let us consider the induction functor 
$\Ind$. It preserves the 
degree functions $\delta_{\tau}$, $\delta_{t_{\tau}}$ on $\Rep(G,\F)$, $\Vect{\F}$ respectively: for any $e\in Epi(\Vect{\F})$, we have \[ \delta_{\tau} (\Ind(e)) = \prod_{I\in Irr(G)} \tau(I)^{[\ker(\Ind(e)):I]} = \prod_{I\in Irr(G)} \tau(I)^{\dim\ker(e) \dim I} = t_{\tau}^{\dim \ker(e)}=\delta_{t_{\tau}}(e). \] 
\end{proof}

The functor $\overline{Ind}$ is described by the universal property of $ \underline{\Rep}(GL_{t_{\tau}}(\F))$ (see \cite{EAH}), stating that each such symmetric monoidal functor out of $\underline{\Rep}(GL_{t_{\tau}}(\F))$ corresponds to an $\F$-linear Frobenius space object of categorical dimension ${t_{\tau}}$ in the target category. This implies in particular that the generating object of $\mathcal{T}(Rep(G,\F),\delta_{\tau})$, which is the target of the tautological ``matrix representation'' of $GL_{t_{\tau}}(\F)$, is an
 $\F$-linear Frobenius space.

\begin{remark}
\cref{cor:ind_res_functoriality_rep_G} has a relative analogue: given a subgroup $H\subset G$, the induction and restriction functors
$$\Ind^G_H\,:\,\Rep(H, \F) \; \leftrightarrows\;\Rep(G, \F)\,: \, \Res^G_H$$ induce symmetric monoidal functors \begin{align*}
    &\overline{\Ind}\;:\;\mathcal{T}(\Rep(H,\F),\delta^H_{\tau'}) \;\longrightarrow \mathcal{T}(\Rep(G,\F),\delta^G_{\tau'_G})\\
    &\overline{\Res}\;:\;\mathcal{T}(\Rep(G,\F),\delta^G_{\tau}) \;\longrightarrow \mathcal{T}(\Rep(H,\F),\delta^H_{\tau_H})
\end{align*}
Here 
\begin{itemize}
    \item for any $\tau:Irr(G)\to \K$ we define $\tau_H:Irr(G)\to \K$ by setting $$\forall \; I'\in Irr(H), \;\;\; \tau_H(I'):=\prod_{I\in Irr(G)} \tau(I)^{[\Res^G_H I: I']}$$
    \item for any $\tau':Irr(H)\to \K$ we define $\tau'_G:Irr(G)\to \K$ by setting $$\forall \; I\in Irr(G), \;\;\; \tau'_G(I):=\prod_{I'\in Irr(H)} \tau'(I')^{[\Res^G_H I: I']}$$
\end{itemize}
and $\delta^H, \delta^G$ denote the corresponding degree functions on $\Rep(H,\F)$, $\Rep(G,\F)$ respectively.
\end{remark}

\section{Adjoint pairs}\label{sec:adjoint-pairs}
\begin{proposition}\label{prop:adjoint}
Let $\mathcal{A},\mathcal{A}'$ be two regular categories and  let $\delta, \delta'$ be degree functions on $\mathcal{A},\mathcal{A}'$ respectively.

Let $F: \mathcal{A}\rightleftarrows \mathcal{A}':G$, $G \vdash F$ be a pair of adjoint functors preserving the degree functions, such that $G$ preserves pullbacks ($F$ automatically preserves pullbacks, see \cref{rmk:adjoint_pullbacks}).

Let $\varepsilon:GF \Rightarrow \id_{\mathcal{A}}$ and $\eta:\id_{\mathcal{A'}}\Rightarrow FG$ be the corresponding counit and unit of the adjunction, and assume that for any two arrows $f: x\to y$ in $\mathcal{A}$, $f': x'\to y'$ in $\mathcal{A}'$ the commutative squares below are Cartesian:
$$ \xymatrix{&x' \ar_{f'}[d]  \ar^-{\eta_{x'}}[r] & FG(x') \ar^{FG(f')}[d] \\&y' \ar^-{\eta_{y'}}[r]  &FG(y')
} \;\;\;\;  \xymatrix{& GF(x) \ar_{GF(f)}[d] \ar^-{\varepsilon_{x}}[r] &x \ar^{f}[d]  \\ &GF(y) \ar^-{\varepsilon_{y}}[r]&y   
}$$

Then the functor $\overline{G}$ is left adjoint to the functor $\overline{F}$.

\end{proposition}

\begin{remark}\label{rmk:adjoint_pullbacks}
Since $G\vdash F$, the functor $G$ automatically preserves colimits (in particular, epimorphisms), and the functor $F$ automatically preserves limits (in particular, pullbacks and monomorphisms).

\end{remark}

\begin{proof}[Proof of \cref{prop:adjoint}]
By \cref{lem:knop_constr_functorial}, the functors $\overline{F}, \overline{G}$ are well-defined.

For any $x\in \mathcal{A}$, $x'\in \mathcal{A}'$, consider the maps $\Gamma_{\mathcal{A}'}(\eta_{x'}): [x'] \to \overline{F}\,\overline{G}[x']$ and $\Gamma_{\mathcal{A}}(\varepsilon_x): \overline{G}\,\overline{F}[x] \to [x]$. Let us show that these maps define natural transformations $\Gamma_{\mathcal{A}'}(\eta): \id_{\mathcal{T}^0(\mathcal{A}',\delta')} \Longrightarrow \overline{F} \,\overline{G}$ and $\Gamma_{\mathcal{A}}(\varepsilon): \overline{G}\,\overline{F}\Longrightarrow \id_{\mathcal{T}^0(\mathcal{A},\delta)} $. Indeed, by \cref{ssec:embedding_orig_cat}, the morphisms in $\mathcal{T}^0(\mathcal{A},\delta) $ are generated by maps of the form $\Gamma_{\mathcal{A}}(f), \Gamma_{\mathcal{A}}(f)^{\vee}$ where $f$ is a morphism in $\mathcal{A}$ (and similarly for $\mathcal{A}'$). So it is enough to check that for any $f:x\to y$ in $\mathcal{A}$ and any $f':x'\to y'$ in $\mathcal{A}'$, the following four diagrams are commutative:
\begin{align*}
   & \xymatrix{&[x'] \ar_{\Gamma_{\mathcal{A}'}(f')}[d]  \ar^-{\Gamma_{\mathcal{A}'}(\eta_{x'})}[rr] && \overline{F}\,\overline{G}[x'] \ar^{\overline{F}\,\overline{G}\Gamma_{\mathcal{A}'}(f')}[d] \\&[y'] \ar^-{\Gamma_{\mathcal{A}'}(\eta_{y'})}[rr]  &&\overline{F}\,\overline{G}[y']
} 
\;\;\;\;
\xymatrix{&[y'] \ar_{\Gamma_{\mathcal{A}'}(f')^{\vee}}[d]  \ar^-{\Gamma_{\mathcal{A}'}(\eta_{y'})}[rr] && \overline{F}\,\overline{G}[y'] \ar^{\overline{F}\,\overline{G}(\Gamma_{\mathcal{A}'}(f')^{\vee})}[d] \\&[x'] \ar^-{\Gamma_{\mathcal{A}'}(\eta_{x'})}[rr]  &&\overline{F}\,\overline{G}[x']
}  \\
 & \xymatrix{&\overline{G}\,\overline{F}[x] \ar_{\overline{G}\,\overline{F}\Gamma_{\mathcal{A}}(f)}[d]  \ar^-{\Gamma_{\mathcal{A}}(\varepsilon_{x})}[rr] && [x] \ar^{\Gamma_{\mathcal{A}}(f)}[d] \\&\overline{G}\,\overline{F}[y] \ar^-{\Gamma_{\mathcal{A}}(\varepsilon_{y})}[rr]  &&[y]
}  
\;\;\;\;
\xymatrix{&\overline{G}\,\overline{F}[y] \ar_{\overline{G}\,\overline{F}(\Gamma_{\mathcal{A}}(f)^{\vee})}[d]  \ar^-{\Gamma_{\mathcal{A}}(\varepsilon_{y})}[rr] && [y] \ar^{\Gamma_{\mathcal{A}}(f)^{\vee}}[d] \\&\overline{G}\,\overline{F}[x] \ar^-{\Gamma_{\mathcal{A}}(\varepsilon_{x})}[rr]  &&[x]
}  
\end{align*} 

Recall that we have natural (canonical) isomorphisms $\Gamma_{\mathcal{A}'}F =\overline{F} \Gamma_{\mathcal{A}} $ and $\Gamma_{\mathcal{A}}G =\overline{G} \Gamma_{\mathcal{A}'} $, which induce natural isomorphisms $$\overline{F}\,\overline{G}\Gamma_{\mathcal{A}'}= \Gamma_{\mathcal{A}'} FG, \;\; \overline{G}\,\overline{F}\Gamma_{\mathcal{A}}= \Gamma_{\mathcal{A}} GF.$$

Thus the commutativity of the two leftmost diagrams follows from the fact that $\eta$ and $\varepsilon$ are natural transformations.

For the rightmost diagrams, we first note that we have natural (canonical) isomorphisms $$\overline{F}((-)^{\vee}) \cong (\overline{F}(-))^{\vee},\;\;\;\;\overline{G}((-)^{\vee}) \cong (\overline{G}(-))^{\vee},$$ so it is enough to check that we have equalities 
\begin{equation}\label{eq:nat_transf_adjoint}
    (\Gamma_{\mathcal{A}'}FG(f'))^{\vee} \circ \Gamma_{\mathcal{A}'}(\eta_{y'}) = \Gamma_{\mathcal{A}'}(\eta_{x'}) \circ \Gamma_{\mathcal{A}'}(f')^{\vee},\;\;\;\; \Gamma_{\mathcal{A}}(f)^{\vee}\circ \Gamma_{\mathcal{A}}(\varepsilon_{y}) =\Gamma_{\mathcal{A}}(\varepsilon_{x}) \circ (\Gamma_{\mathcal{A}}GF(f))^{\vee}. 
\end{equation}

Next, recall from \cref{ssec:embedding_orig_cat} that $\Gamma {\beta}^{\vee}\circ \Gamma {\alpha} = \Gamma {\alpha'}\circ \Gamma {\beta'}^{\vee}$, for any Cartesian square $$  \xymatrix{&{\bullet} \ar_{\beta'}[d]  \ar^{\alpha'}[r] & {\bullet} \ar^{\beta}[d] \\&{\bullet} \ar^{\alpha}[r]  &{\bullet}   
}$$ either in $\mathcal{A}$ or in $\mathcal{A}'$, with $\Gamma$ denoting $\Gamma_{\mathcal{A}}$ or $\Gamma_{\mathcal{A}'}$ respectively. Equalities \eqref{eq:nat_transf_adjoint} now follow from the assumptions on $\varepsilon, \eta$.

Thus $\Gamma_{\mathcal{A}'}(\eta)$, $\Gamma_{\mathcal{A}}(\varepsilon)$ are natural transformations. 

We now check the ``snake relations'' required to make them unit and counit of adjunction $\overline{G}\vdash \overline{F}$. For any $x\in \mathcal{A}, x'\in \mathcal{A}'$ we need to show that the compositions below are identity maps:
\begin{align*}
&\overline{F}[x] \xrightarrow{\;(\Gamma_{\mathcal{A}'}\eta) \overline{F}\;\;} \overline{F}\,\overline{G}\,\overline{F}[x] \xrightarrow{\;\overline{F}( \Gamma_{\mathcal{A}} \varepsilon) \;\;} \overline{F}[x]
\\
&\overline{G}[x'] \xrightarrow{\;\overline{G}(\Gamma_{\mathcal{A}'}\eta)\;\; } \overline{G}\,\overline{F}\,\overline{G}[x'] \xrightarrow{\; ( \Gamma_{\mathcal{A}} \varepsilon) \overline{G} \;\;} \overline{G}[x']
\end{align*}
Yet these compositions are just given by \begin{align*} \Gamma_{\mathcal{A}'}(F\varepsilon_x)\circ \Gamma_{\mathcal{A}'}(\eta_{F(x)}) & = \Gamma_{\mathcal{A}'}(\id_{F(x)}) = \id_{\overline{F}[x]} \;\; \text{ and } \\\Gamma_{\mathcal{A}}(\varepsilon_{G(x')})\circ \Gamma_{\mathcal{A}}(G(\eta_{x'})) & = \Gamma_{\mathcal{A}}(\id_{G(x')}) = \id_{\overline{G}[x']}\end{align*} 
as required.
\end{proof}

An immediate corollary of \cref{prop:adjoint} is:
\begin{corollary}
Let $\mathcal{A}, \mathcal{A}', F, G$ be as in \cref{prop:adjoint}. Then the induced $\K$-linear symmetric monoidal functor $\overline{F}:\mathcal{T}(\mathcal{A},\delta) \longrightarrow \mathcal{T}(\mathcal{A}',\delta')$ preserves limits (those which exist in $\mathcal{T}(\mathcal{A},\delta)$). 
\end{corollary}

\section{Examples of adjoint pairs and their induced functors}\label{sec:examples_adj_pairs}
\subsection{Overcategory (slice category)}\label{ssec:slice_cat}
Let $\mathcal{A}$ be a regular category and let $\delta$ be a degree function on $\mathcal{A}$. 

Let $x_0\in \mathcal{A}$ and let $\mathcal{A}/x_0$ be the slice category (overcategory) of $\mathcal{A}$ over $x_0$: its objects are pairs $(x, p_x)$ where $x\in \mathcal{A}$, $p_x:x\to x_0$, and a morphism $(x,p_x)\to (y,p_y)$ in $\mathcal{A}/x_0$ is given by $f: x\to y$ such that $p_y\circ f=p_x$. 

By \cite[Example 3, Section 2]{K}, the category $\mathcal{A}/x_0$ is regular as well. 
Consider the induced degree function $\delta'$ on $\mathcal{A}/x_0$:
given $e:(x, p_x)\twoheadrightarrow (y, p_y)$ in $Epi(\mathcal{A}/x_0)$, we let $\delta'(e):=\delta(e)$ where the latter $e$ denotes the epimorphism $x\twoheadrightarrow y$ in $Epi(\mathcal{A})$.

We have obvious functors
$F : \mathcal{A}\rightleftarrows \mathcal{A}/x_0:G$ where $F:= (-) \times x_0$ and $G$ is the forgetful functor $(x, p_x)\mapsto x$. Clearly, $G$ is left adjoint to $F$, both functors preserve the degree functions $\delta, \delta'$, and $G$ preserves pullbacks.

\begin{corollary}\label{cor:slice_cat}
The functor $F: \mathcal{A}\to \mathcal{A}/x_0$ preserves finite limits and the induced $\K$-linear symmetric monoidal functor $\overline{F}:\mathcal{T}(\mathcal{A},\delta) \longrightarrow \mathcal{T}(\mathcal{A}',\delta')$ preserves finite limits. 
\end{corollary}
\begin{proof}
    We need to check the Cartesian square conditions appearing in \cref{prop:adjoint}. Indeed, let $x\in \mathcal{A}$. The counit $\varepsilon_x: x\times x_0\to x$ is just the projection onto the first factor and for any $f:x\to y$, the square $\xymatrix{& x\times x_0 \ar_{f\times \id_{x_0}}[d] \ar^-{\varepsilon_{x}}[r] &x \ar^{f}[d]  \\ &y \times x_0 \ar^-{\varepsilon_{y}}[r]&y   
} $ is clearly Cartesian.
Similarly, for any $x\in \mathcal{A}$, $p_x:x\to x_0$, the unit $\eta_{(x, p_x)}: (x, p_x)\to (x\times x_0,  pr_{x_0}:x\times x_0 \to x_0)$ is just $(\id, p_x)$. So for any map $f:(x, p_x)\to (y, p_y)$ in $\mathcal{A}/x_0$, the square
 $$\xymatrix{&(x, p_x) \ar_{f}[d]  \ar^-{\eta_{(x, p_x)}}[rr] & & (x\times x_0, pr_{x_0}) \ar^{f\times \id_{x_0}}[d] \\&(y, p_y) \ar^-{\eta_{(y, p_y)}}[rr] & &(y\times x_0, pr_{x_0})
}$$ is clearly Cartesian as well.
\end{proof}
\begin{remark}
    If we replace $\mathcal{A}/x_0$ by the category whose objects are pairs $(x, p_x)$ where $x\in \mathcal{A}$, $p_x:x\twoheadrightarrow x_0$ is surjective, then the forgetful functor $G$ will not preserve pullbacks, and the functor $\overline{G}$ will not necessarily be defined on $\mathcal{T}^0(\mathcal{A}, \delta)$, e.g. if $\mathcal{A}=\Vect{\F_q}$.
\end{remark}

We apply this construction in the case $\mathcal{A}:=\Vect{\F_q}$.
 
Fix $t\in \mathbbm{k}$ and $V_0 \in \Vect{\F_q}$, and consider the overcategory $\Vect{\F_q}/V_0$.  We fix the degree function $\delta_t$ on $\Vect{\F_q}$ given by $\delta_t(e):=t^{\dim \ker(e)}$ for $e\in Epi(\Vect{\F_q})$. 

\begin{remark}\label{remark-max-par}
The Karoubi symmetric monoidal category
$$\mathcal{T}(\Vect{\F_q}/V_0, \delta'_t)$$ interpolates categories of representations of groups $\Aut(V, p:V\to V_0)$. Such groups can be identified with maximal parabolic subgroups of $GL_n(\F_q)$ for different values of $n$, where the parabolic subgroup preserves an $\F_q$-subspace of $\F_q^n$ whose dimension is at most $\dim_{\F_q} V_0$.
\end{remark}

Let $F: \Vect{\F_q} \to \Vect{\F_q}/V_0$ be given by $F:=(-)\oplus V_0$.
By \cref{cor:slice_cat}, we have:
\begin{corollary}
The symmetric monoidal functor
$$\overline{F}: \underline{\Rep}(GL_{t}(\F_q)) \to  \mathcal{T}(\Vect{\F_q}/V_0, \delta'_t)$$
preserves finite limits.
\end{corollary}

\comment{
}
\subsection{Iterated wreath products and symmetric groups}\label{ssec:functoriality_iterated_wreath}

We have a triple of adjoint functors $F\vdash T\vdash (-)^G$, where 
\begin{itemize}
    \item $T: \Sets \to G-\Sets$ is given by sending a set to the same set with the trivial $G$-action,
    \item $F:=(-)/G: G-\Sets\to \Sets$ is given by taking the set of orbits of a given $G$-set, and 
    \item $ (-)^G:G-\Sets \to \Sets$ is the functor of taking $G$-invariants, represented by $\Hom_G(\{\bullet\},-)$ where $\{\bullet\}$ is the singleton set with trivial $G$-action.
\end{itemize}

The above functors induce functors $$T:\Sets^{op}\to G-\Sets^{op}, \;\; F, (-)^G :G-\Sets^{op}\to \Sets^{op}$$ and these functors are adjoint: $(-)^G\vdash T\vdash F$. 

\begin{remark}
The functor $(-)^G:G-\Sets^{op}\to \Sets^{op}$ does not preserve pullbacks, since $(-)^G:G-\Sets\to \Sets$ doesn't preserve pushouts. For example, letting $G$ stand for the the free transitive $G$-set, we have: $$\xymatrix{& G \ar[d] \ar[r] &\bullet \ar[d] &\ar_{(-)^G}@{|->}[rr] && & \emptyset \ar[d] \ar[r] &\bullet \ar[d]  \\ &\bullet \ar[r]&\bullet    & & &
&\bullet \ar[r]&\bullet   
} $$ The left square is coCartesian in $G-\Sets$, while the right square is not co-Cartesian in $\Sets$.
\end{remark}

Fix $t\in \K$ and let $ \phi_t:Tran(G)\to \K$ be the constant function $\phi_t \equiv t$. Let $\delta_t$, $\delta_{\phi_t}$ be the corresponding degree functions on $\Sets^{op}$ and $G-\Sets^{op}$ respectively.

\begin{lemma} The pair of adjoint functors $(T, F)$ between $(\Sets^{op}, \delta_t)$ and $(G-\Sets^{op}, \delta_{\phi_t})$ preserve pullbacks, terminal objects and degree functions. They give rise to two symmetric monoidal functors $$\overline{T}\,:\,\underline{\Rep}(S_t)\; \rightleftarrows \;\mathcal{T}(G-\Sets^{op}, \delta_{\phi_t}) \,:\,\overline{F}$$
so that $\Gamma_{\Sets^{op}}(\eta) :  \id\xrightarrow{\sim} \overline{F}\,\overline{T}$ where $\eta$ is the unit of the adjunction $T\vdash F$.
\end{lemma}

\begin{proof} The functors $T:\Sets^{op}\leftrightarrows G-\Sets^{op}:F$ have left adjoints so they automatically preserve pullbacks and epimorphisms, and they clearly preserve the terminal object (the empty set) as well. Concerning the degree functions, let $e:X\hookrightarrow Y$ be a monomorphism in $\Sets$ and $e':X'\hookrightarrow Y'$ be a monomorphism in $G-\Sets$. Then $$\delta_{\phi_t}(T(e))=t^{\abs{Y\setminus X}} = \delta_t(e) \;\; \text{ and } \;\; \delta_t(F(e')) = t^{|(Y/G) \,  \setminus \, (X/G)|} = t^{|(Y\setminus X)/G|}  = \delta_{\phi_t}(e').$$

The composition $F\circ T$ is canonically isomorphic to the identity endofunctor of $\Sets^{op}$ and the unit of the adjunction $T\vdash F$ is the canonical isomorphism 
$\eta=\id:\id_{\Sets^{op}} \to F\circ T$. As in the proof of \cref{prop:adjoint}, this immediately implies the statement of the lemma.
\end{proof}

\begin{remark}
The composition $T\circ F$ is the endofunctor on $G-\Sets^{op}$ sending a $G$-set $X$ to the set $X/G$ with the trivial $G$-action, and the counit of the adjunction $T\vdash F$ is given by the obvious maps $\varepsilon_X:X\to X/G$ in $G-\Sets$, for every $G$-set $X$.

The conditions in the statement of \cref{prop:adjoint} do not hold: while the required commutative squares for the unit $\eta$ are obviously Cartesian, the commutative squares for the counit $\varepsilon$ might not be Cartesian if the corresponding map $f\in \Hom_{G-\Sets^{op}}(X, Y)$ is not a monomorphism (take for example $Y:=\emptyset$, $X$ a free transitive $G$-set, and $f: \emptyset\to X$ the trivial map).

In fact, it is easy to verify that $\Gamma_{G-\Sets^{op}}(\varepsilon)$ doesn't define a counit of adjunction for $(\overline{T}, \overline{F})$. 

\end{remark}

\subsection{Pullbacks between iterated wreath products}
Let $N\triangleleft G$ be a normal subgroup of $G$, and let $q: G\twoheadrightarrow G/N$ be the quotient homomorphism. Similarly to the discussion above, we have a triple of functors $$q^*:G/N-\Sets^{op}\to G-\Sets^{op}, \;\; (-)/N, (-)^N :G-\Sets^{op}\to G/N-\Sets^{op}$$ given by pullback with respect to $q$, taking $N$-orbits and $N$-invariants respectively. These functors are again adjoint: $(-)^N\vdash q^*\vdash (-)/N$. Hence  $q^*, (-)/N $ preserve pullbacks and terminal objects. 

As in \cref{ssec:functoriality_iterated_wreath}, the composition of the functors $(-)/N \circ q^*$ is canonically isomorphic to the identity functor.

Next, the functor $(-)/N$ defines a map $\kappa:Tran(G)\to Tran(G/N)$.
Let $\phi:Tran(G/N)\to \K$, and consider the degree functions $\delta_{\phi}$, $\delta_{\phi\circ \kappa}$ on $G/N-\Sets^{op}$ and $G-\Sets^{op}$ respectively.

\begin{lemma} The pair of adjoint functors $q^*, (-)/N$ give rise to two symmetric monoidal functors between the iterated wreath product categories:
$$\overline{(-)^N}\;:\;\mathcal{T}(G-\Sets^{op}, \delta_{\phi\circ \kappa})\;\; \leftrightarrows \;\;\mathcal{T}(G/N-\Sets^{op}, \delta_{\phi})\;:\;\overline{q^*}.$$
so that $\Gamma_{G/N-\Sets^{op}}(\eta) :  \id\xrightarrow{\sim} \overline{\overline{(-)^N}}\,\overline{q^*}$ where $\eta$ is the unit of the adjunction $q^*\vdash (-)^N$.
\end{lemma}



\end{document}